\documentclass[12pt,oneside]{article}
      \usepackage{amsmath}
      \usepackage{amsthm}
      \usepackage{amscd}
      \usepackage{amssymb}
            \usepackage{epsfig}
        \usepackage{graphicx}
            \usepackage{latexsym}
            \usepackage{epsfig}
            \usepackage{epic}
            \usepackage{color}
      \usepackage{subfigure}
      \usepackage{verbatim}

      \theoremstyle{plain}
      \newtheorem{theorem}{Theorem}[section]
      \newtheorem{lemma}[theorem]{Lemma}
      \newtheorem{corollary}[theorem]{Corollary}
      
      \theoremstyle{definition}
      
      \theoremstyle{remark}
      
      \long\def\symbolfootnote[#1]#2{\begingroup%
      \def\thefootnote{\fnsymbol{footnote}}\footnote[#1]{#2}\endgroup} 

\numberwithin{equation}{section}
\numberwithin{conj}{section}
\numberwithin{figure}{section}

\usepackage{anysize}
\marginsize{2cm}{2cm}{2cm}{3cm}

\begin{document}

\title{When are two Dedekind sums equal? }
\author{Stanislav Jabuka, Sinai Robins and Xinli Wang}

\date{ }

\maketitle
\begin{abstract}
A natural question about Dedekind sums is to find conditions on
the integers $a_1, a_2$, and $b$ such that $s(a_1,b) = s(a_2, b)$.
We prove that if the former equality holds then  $ b \ | \ (a_1a_2-1)(a_1-a_2)$.
Surprisingly, to the best of our knowledge such statements have not appeared in the literature.
A similar theorem is proved for the more general Dedekind-Rademacher sums as well, 
namely that for any fixed non-negative integer $n$, a positive integer modulus $b$, and two integers $a_1$ and $a_2$ 
that are relatively prime to $b$, the hypothesis $r_n (a_1,b)= r_n (a_2,b)$ implies that 
 $b \  |  \    (6n^2+1-a_1a_2)(a_2-a_1)$.
\end{abstract}

\symbolfootnote[0]{2010 Mathematics Subject Classification:11F20.
\indent \indent The second author is supported in part by the SUG grant M5811053,
from the Nanyang Technological University. The first author was partially supported by the NSF grant DMS-0709625.
}

\section{Introduction}
\bigskip

\indent \indent Dedekind sums arise naturally in many fields, most prominently in combinatorial geometry \cite{Beck} and in the theory of modular forms \cite{Rademacher}.
 The classical {\bf Dedekind sum} is defined by:
\[
s(a,b) = \sum_{k=0}^{b-1}{\left(\left(\frac{ka}{b}\right)\right)\left(\left(\frac{k}{b}\right)\right)},
\]
where $a$ and $b$ are any two relatively prime integers, and where the {\bf Sawtooth function} is defined by
\begin{equation*}
((x))=
\begin{cases} \{x\} - \frac{1}{2} & \text{if $x \notin \mathbb Z$,}
\\
0 &\text{if $x \in \mathbb Z$.}
\end{cases}
\end{equation*}

The Dedekind sum enjoys two important properties.  The first of these properties is the periodicity of the Dedekind sums in the
first variable, namely $s(a + kb,b) = s(a,b) \mbox{ for all } k \in \mathbb Z$. The second, and deeper, of these properties is the famous reciprocity law for Dedekind sums:
\[
s(a,b)+s(b,a)  = -\frac{1}{4}+\frac{1}{12}\left(\frac{a}{b}+
\frac{1}{ab}+\frac{b}{a} \right),
\]
valid for any two relatively prime integers  $a$ and $b$.
It is very natural to ask under what conditions on the integers $a_1, a_2$, and $b$ is it true that
\[
s(a_1,b) = s(a_2, b)?
\]

\noindent
This question also arises in topological considerations, and involves the correction terms of the Heegaard Floer Homology \cite{Steven}. We answer this question with the following results.

\begin{theorem}\label{Dedekind}
Let $b$ be a positive integer, and $a_1$, $a_2$ any two integers that are relatively prime to $b$.   If $s(a_1,b)=s(a_2,b)$, then
\[
b|(1-a_1a_2)(a_1-a_2).
\]
\end{theorem}

\noindent
An immediate corollary of this theorem is the following result:
\begin{corollary}\label{coro1}
Let $p$ be a prime.    Then  $s(a_1,p) = s(a_2,p)$ if and only if $a_1 \equiv a_2  \bmod p$, or $a_1 a_2 \equiv 1 \bmod p$.
\end{corollary}

\bigskip 

We note that the converse of Theorem \ref{Dedekind} is false in general.   Consider, for example,
$b=40$, and $a_1=37, a_2=33$.   Then $b \ | \ (1-37\cdot33)(37-33) = 20\cdot4$, and yet
$s(37,40)=-\frac{13}{16}$ and $s(33,40)=-\frac{5}{16}$, so that $s(a_1, b) \not= s(a_2, b)$ in this case.

\bigskip
\noindent
We also study the analogous question for the Dedekind-Rademacher sums, which arise in Donald Knuth's work on pseudo-random number generators.   Given any
non-negative integer $n$, and any two relatively prime integers $a$ and $b$, we define the Dedekind-Rademacher sum by:
\[
r_n(a,b) = \sum_{k=0}^{b-1}{\left(\left(\frac{ka+n}{b}\right)\right)\left(\left(\frac{k}{b}\right)\right)}.
\]

In order to state the corresponding reciprocity law for the Dedekind-Rademacher sums, we define
\begin{equation*}
\chi_a(n)=
\begin{cases} 1 & \text{if $a|n$,}
\\
0 &\text{otherwise.}
\end{cases}
\end{equation*}

\begin{lemma}[Reciprocity law for Dedekind-Rademacher sums]\label{rec12}
Let $a$ and $b$ be relatively prime positive integers. Then for $n=1,2,\cdots,a+b,$
\begin{align*}
r_n(a,b)+r_n(b,a)&= \frac{n^2}{2ab}-\frac{n}{2}\left(\frac{1}{a}+ \frac{1}{b} + \frac{1}{ab}\right)+\frac{1}{12}\left(\frac{b}{a}+ \frac{a}{b}+ \frac{1}{ab} \right)\\
&+ \frac{1}{2}\left( \left(\left(\frac{a^{-1}n}{b}\right)\right)+\left(\left(\frac{b^{-1}n}{a}\right)\right)+\left(\left(\frac{n}{a}\right)\right)+\left(\left(\frac{n}{b}\right)\right) \right)\\
&+ \frac{1}{4}\left(1+\chi_a(n) + \chi_b(n)\right),
\end{align*}
where $aa^{-1} \equiv 1 \bmod b$ and $bb^{-1} \equiv 1 \bmod a.$
\end{lemma}

\noindent
The proof of Lemma \ref{rec12} can be found, for example,  in \cite{Beck}.
For the Dedekind-Rademacher sums, we have the following two results.

\begin{theorem}\label{Dedekind-Rademacher}
Fix a non-negative integer $n$ and a positive integer $b$.  Let $a_1$ and $a_2$ be any two integers that are relatively prime to $b$.

If $r_n (a_1,b)= r_n (a_2,b)$, then
\[
b \  |  \    (6n^2+1-a_1a_2)(a_2-a_1).
\]
\end{theorem}

\bigskip
\noindent
Again, an immediate Corollary for prime moduli follows.
\begin{corollary}\label{coro 2}
Let $p$ be a prime.    If $r_n (a_1, p) = r_n (a_2, p)$, then it follows that  $a_1 \equiv a_2 \bmod p$, or
 \[
 a_1a_2 \equiv 1+6n^2 \bmod p.
 \]
\end{corollary}

Note that the converses of Theorem \ref{Dedekind-Rademacher} and Corollary \ref{coro 2} are generally false. A counter example is provided by $a_1=3$, $a_2=11$, $b=23$ and $n=6$ for which $b|(6n^2+1-a_1a_2)$ but $r_6(3,23)=-\frac{3}{92}$ while $r_6(11,23) = \frac{43}{92}$. 


As a direction for further research, we note that when $b$ is
composite and $c$ is rational,  the equation $s(x,b)=c$ might have
more than $2$ solutions in $x \in \mathbb Z$. In fact, Corollary
\ref{coro1} shows that if $b$ has $r$ distinct prime divisors,
then  the number of solutions to $s(x,b) = c$ is greater than or
equal to  $2^r$, by the usual elementary modular arithmetic
arguments. It would be quite interesting to study how many integer
solutions in $x \in \mathbb Z$ the equation $s(x,b) = c$ has in
general.

\section{Proofs}

\bigskip
We first  introduce some lesser-known but useful properties of Dedekind sums.   It is proved in \cite{Rademacher} that
\begin{equation}\label{identity}
6b \ s(a,b) \in \mathbb Z,
\end{equation}
for any two relatively prime integer $a$ and $b$. This property of Dedekind sums  gives us a nice upper bound on the denominators
that any Dedekind sum $s(a,b)$ may have,  and it plays an interesting role in the proof of Theorem \ref{Dedekind}.


\bigskip

\begin{proof}[Proof of Theorem \ref{Dedekind}]

For any integers $a_1$ relatively prime to $b$, and $a_2$ relatively prime to $b$, Dedekind's Reciprocity law implies that we
have the following two identities:
\begin{align}
12a_1b \left(s(a_1,b)+s(b,a_1) \right) &= -3a_1b+a_1^2+b^2+1,\label{rec1} \\
12a_2 b \left(s(a_2,b)+s(b,a_2)  \right) &= -3a_2b+a_2^2+b^2+1. \label{rec2}
\end{align}

\noindent
Multiplying \eqref{rec1} with $a_2$, and multiplying  \eqref{rec2} with $a_1$,  we get
\begin{align}
12a_1a_2 b  (s(a_1,b)+s(b,a_1)) &= a_2\left(-3a_1b+a_1^2+b^2+1\right),\label{rec3} \\
12a_1a_2 b (s(a_2,b)+s(b,a_2)) &= a_1\left(-3a_2b+a_2^2+b^2+1\right).
\label{rec4}
\end{align}

\noindent
Subtracting  $\eqref{rec4} $ from $\eqref{rec3} $ gives us
\begin{align}
 & 12a_1a_2 b (s(a_1,b)+s(b,a_1))-12a_1a_2b (s(a_2,b)+s(b,a_2)) \\ \nonumber
&= a_2\left(-3a_1b+a_1^2+b^2+1
\right)-a_1\left(-3a_2b+a_2^2+b^2+1\right).
\end{align}

\noindent  We know, by assumption, that $s(a_1,b) = s(a_2,b)$, and therefore
\begin{equation}\label{eq}
12a_1a_2 b s(b,a_1)-12a_1a_2 b s(b,a_2)=
a_1^2a_2+b^2a_2-b^2a_1+a_2-a_2^2a_1-a_1.
\end{equation}
Using the fact (\ref{identity}) that $(6   a_1)  s(b, a_1) $ and $ (6 a_2)  s(b, a_2)$ are both integers, we may
reduce \eqref{eq}  mod $b$ to obtain the result
\begin{equation}
(a_2-a_1)(1-a_1a_2)\equiv 0 \bmod b.
\end{equation}
\end{proof}

\begin{lemma}\label{int2}
For any relatively prime integers $a$ and $b$, we have  $$12b  \ r_n(a,b) \in \mathbb Z.$$
\end{lemma}

\begin{proof}
Note that for relatively prime numbers $a$ and $b$,there's always a solution to the equation $ka+n \equiv 0 \pmod b$, while $k \in \{0,1,\cdots,b-1\}.$
We consider two different situations.
\begin{itemize}
\item When $k=0,$ or equivalently if $n \equiv 0 \pmod b$, then $r_n(a,b) = r_0(a,b) = s(a,b)$ and it was pointed out in \cite{Rademacher} that $6b \, s(a,b) \in \mathbb Z$. 
\item When $k=k_0a+n \equiv 0 \pmod b$ where $k_0 \in \{1,\cdots,b-1\}.$
\begin{align*}
12b r_n({a,b}) &= 12b \sum_{k=0}^{b-1}{\left(\left(\frac{ka+n}{b}\right)\right)\left(\left(\frac{k}{b}\right)\right)},\\
&= 12b \sum_{k=1,k \neq k_0}^{b-1}{\left( \frac{ka+n}{b}  - \left[\frac{ka+n}{b}\right]-\frac{1}{2} \right)\left(  \frac{k}{b} - \frac{1}{2}\right)},\\
&= 12b \sum_{k=1, k \neq k_0}^{b-1}\left(\frac{k(ka+n)}{b^2}-\frac{A}{2b}+\frac{1}{4}\right),\\
&= 12b \left( \frac{a(b-1)(2b-1)}{6b}+\frac{n(b-1)}{2b}-\frac{A(b-2)}{2b} + \frac{b-2}{4} - \frac{Ck_0}{b} \right),\\
&= 2a(b-1)(2b-1) + 6n(b-1)- 6A(b-2) + 3b(b-2) - 12Ck_0.
\end{align*}
where $A,C \in \mathbb Z,$ and immediately we have $12b r_n(a,b) \in \mathbb Z.$
\end{itemize}

\end{proof}

\bigskip

\begin{proof}[Proof of Theorem \ref{Dedekind-Rademacher}]

From the  reciprocity law for the Dedekind-Rademacher sums, we know that when $a_1,a_2$
 are relatively prime to $b$, we have:
\begin{align}
12a_1b \left(r_n(a_1,b)+r_n(b,a_1)\right) \nonumber &=
6n^2+a_1^2+b^2+1-9a_1b\\ \nonumber
&-6a_1b\left(\left[\frac{a_1^{-1}n}{b}\right]
+\left[\frac{b^{-1}n}{a_1}\right]
+\left[\frac{n}{b}\right]+\left[\frac{n}{a_1}\right] \right) \\
\label{eq3} &+ 3a_1b(\chi_{a_1}(n)+\chi_{b}(n)),\\ \nonumber
\\\
12a_2b\left(r_n(a_2,b)+r_n(b,a_2)\right) \nonumber &=
6n^2+a_2^2+b^2+1-9a_2b\\ \nonumber
&-6a_2b\left(\left[\frac{a_2^{-1}n}{b}\right]
+\left[\frac{b^{-1}n}{a_2}\right]
+\left[\frac{n}{b}\right]+\left[\frac{n}{a_2}\right] \right) \\
\label{eq4} &+ 3a_2b(\chi_{a_2}(n)+\chi_{b}(n)).\\ \nonumber
\end{align}

To simplify the ensuing algebra, we let
\[
S_{a_1} = \left[\frac{a_1^{-1}n}{b}\right]
+\left[\frac{b^{-1}n}{a_1}\right]
+\left[\frac{n}{b}\right]+\left[\frac{n}{a_1}\right] \in \mathbb Z,
\]
\[
S_{a_2} = \left[\frac{a_2^{-1}n}{b}\right]
+\left[\frac{b^{-1}n}{a_2}\right]
+\left[\frac{n}{b}\right]+\left[\frac{n}{a_2}\right] \in \mathbb Z,
\]
\[
T_{a_1} = \chi_{a_1}(n)+\chi_{b}(n) \in \mathbb Z,
\]
\[
T_{a_2} = \chi_{a_2}(n)+\chi_{b}(n) \in \mathbb Z,
\]
we can rewrite \eqref{eq3} and \eqref{eq4} as follows:
\begin{align}
 12a_1b\left(r_n(a_1,b)+r_n(b,a_1)\right) \label{eq5}
  &= 6n^2+a_1^2+b^2+1-9a_1b-6a_1bS_{a_1}+3a_1bT_{a_1},\\ \nonumber
  \\\
 12a_2b\left(r_n(a_2,b)+r_n(b,a_2)\right)
  &= 6n^2+a_2^2+b^2+1-9a_2b-6a_2 bS_{a_2}+3a_2 bT_{a_2}. \label{eq6}
 \end{align}

\noindent Multiplying \eqref{eq5} with $a_2$ gives us
\begin{align}
& \label{eq7} 12a_1a_2b \left(r_n(a_1,b)+r_n(b,a_1)\right) \\
\nonumber
  &= a_2\left(6n^2+a_1^2+b^2+1-9a_1b-6a_1b S_{a_1}+3a_1bT_{a_1}\right).
\end{align}

\noindent Multiplying \eqref{eq6} with $a_1$ gives us
\begin{align}
& \label{eq8} 12a_1a_2b \left(r_n(a_2,b)+r_n(b,a_2)\right)\\
\nonumber
  &= a_1\left(6n^2+a_2^2+b^2+1-9a_2b-6a_2bS_{a_2}+3a_2bT_{a_2}\right).
\end{align}

\noindent Since $r_n(a_1,b) = r_n(a_2,b)$ ,
subtracting $\eqref{eq8}$ from $\eqref{eq7}$ we get
\begin{align}
& \label{eq9}12a_1a_2 b \left(r_n(b,a_1) - r_n(b,a_2)\right) \\
\nonumber &= (a_2- a_1)(6n^2+1-a_1a_2)+b^2(a_2-a_1)\\ \nonumber
&-6a_1a_2b (S_{a_1}+S_{a_2})+3a_1a_2b (T_{a_1}+T_{a_2}). \nonumber
\end{align}

We notice that, by Lemma  \eqref{identity}, we have  $12a_1r_n(b,a_1)\in \mathbb Z$ and $12a_2r_n(b,a_2) \in \mathbb Z$.
We may therefore reduce both sides of   \eqref{eq9} modulo $b$ to obtain the result:
\begin{equation*}
0 \equiv (6n^2+1-a_1a_2)(a_2-a_1) \bmod b.
\end{equation*}
\end{proof}

\bigskip

\noindent
Department of Mathematics and Statistics,\\
University of Nevada, Reno.\\
  \texttt{jabuka@unr.edu}

 \medskip \noindent
Division of Mathematical Sciences, \\
School of Physical and Mathematical Sciences,\\
Nanyang Technological University, Singapore, 637371 \\
  \texttt{rsinai@ntu.edu.sg}

\medskip \noindent
Division of Mathematical Sciences, \\
School of Physical and Mathematical Sciences,\\
Nanyang Technological University, Singapore, 637371 \\
 \texttt{wang0480@e.ntu.edu.sg}

\end{document}